\documentclass[a4paper]{amsart}

\usepackage{a4wide}
\usepackage[colorlinks, final]{hyperref}
\usepackage[utf8]{inputenc}

\usepackage{amsmath, amscd, amsfonts, amsthm, amssymb, bbm, mathtools}
\usepackage{aliascnt}
\usepackage{enumitem}
\usepackage{float, graphicx, tikz}
\usepackage{xkvltxp, fixme}
\usepackage{xfrac}

\mathtoolsset{showonlyrefs}
\fxsetup{targetlayout=color}

\makeatletter

\title{Irreducibility of Polynomials with Square Coefficients over Finite Fields}

\author{Lior Bary-Soroker}

\address{Raymond and Beverly Sackler School of Mathematical Sciences, Tel Aviv University, Tel Aviv}
\email{barylior@tauex.tau.ac.il}

\author{Roy Shmueli}

\address{Raymond and Beverly Sackler School of Mathematical Sciences, Tel Aviv University, Tel Aviv}
\email{royshmueli@mail.tau.ac.il}

\renewcommand{\pod}[1]{\allowbreak\mathchoice
  {\if@display \mkern 18mu\else \mkern 8mu\fi (#1)}
  {\if@display \mkern 18mu\else \mkern 8mu\fi (#1)}
  {\mkern4mu(#1)}
  {\mkern4mu(#1)}
}

\newcommand{\autotheorem}[3]{
\newaliascnt{#1counter}{#2}
\newtheorem{#1}[#1counter]{#3}
\expandafter\newcommand\csname #1counterautorefname\endcsname{#3}
}

\theoremstyle{plain}

\autotheorem{theorem}{basetheorem}{Theorem}
\autotheorem{lemma}{basetheorem}{Lemma}
\autotheorem{proposition}{basetheorem}{Proposition}
\autotheorem{corollary}{basetheorem}{Corollary}
\autotheorem{fact}{basetheorem}{Fact}
\autotheorem{conjecture}{basetheorem}{Conjecture}

\theoremstyle{definition}
\autotheorem{definition}{basetheorem}{Definition}
\autotheorem{example}{basetheorem}{Example}

\theoremstyle{remark}
\autotheorem{remark}{basetheorem}{Remark}

\newtheoremstyle{case}%
  {\topsep}%
  {\topsep}%
  {}%
  {\parindent}%
  {\itshape}%
  {}%
  { }%
  {\thmname{#1}\thmnumber{ #2}:{\thmnote{ #3.}}}

\theoremstyle{case}

\counterwithin*{case}{basetheorem}

\hypersetup{
  pdftitle={\@title},
  pdfauthor={\@author}
}

\makeatother

\newcommand*{\dchebotarv}[1][1]{\delta_{\mathrm{Cheb}, #1}}

\newcommand*{\rv}{\xi}

\newcommand*{\veca}{{\underline{\alpha}}}
\newcommand*{\vecb}{{\underline{\beta}}}

\newcommand*{\vect}{{\underline{t}}}

\newcommand*{\vecrv}{{\underline{\rv}}}

\newcommand*{\vecs}{{\underline{s}}}

\DeclareMathOperator{\sign}{sgn}

\newcommand*{\disc}[1][x]{\operatorname{Disc}_{#1}}

\newcommand*{\Sym}[1][n]{\operatorname{Sym}\pa*{#1}}
\newcommand*{\Alt}[1][n]{\operatorname{Alt}\pa*{#1}}

\newcommand*{\frobC}[2][\Ocl_K/\Ocl_F]{%
  \mathchoice{\pa*{\frac{#1}{#2}}}%
  {\pa*{\tfrac{#1}{#2}}}%
  {Artin symbol should not be in a subscript/superscript}%
  {Artin symbol should not be in a subscript/superscript}%
}
\newcommand*{\frob}[2][\Ocl_K/\Ocl_F]{%
  \mathchoice{\br*{\frac{#1}{#2}}}%
  {\br*{\tfrac{#1}{#2}}}%
  {Artin symbol should not be in a subscript/superscript}%
  {Artin symbol should not be in a subscript/superscript}%
}

\DeclareMathOperator{\Res}{Res}
\DeclareMathOperator{\cmp}{cmp}

\DeclareMathOperator{\irreg}{irreg}
\DeclareMathOperator{\SplitType}{SplitType}

\newcommand*{\ZZ}{{\mathbb{Z}}}

\newcommand*{\CC}{{\mathbb{C}}}

\newcommand*{\FF}{{\mathbb{F}}}

\renewcommand*{\Pr}{{\mathbb{P}}}
\newcommand*{\PP}{{\mathbb{P}}}

\newcommand*{\cond}{\;\middle|\;}

\newcommand*{\onebb}{{\mathbbm{1}}}

\DeclarePairedDelimiter{\pa}{\lparen}{\rparen}
\DeclarePairedDelimiter{\br}{\lbrack}{\rbrack}

\DeclarePairedDelimiter{\set}{\{}{\}}
\DeclarePairedDelimiter{\abs}{\lvert}{\rvert}

\newcommand*{\Ccl}{\mathcal{C}}

\newcommand*{\Fcl}{\mathcal{F}}

\newcommand*{\Ocl}{\mathcal{O}}

\newcommand*{\Scl}{\mathcal{S}}

\DeclareMathOperator{\Gal}{Gal}
\DeclareMathOperator{\Hom}{Hom}

\renewcommand{\AA}{\mathbb{A}}

\begin{document}

\begin{abstract}
We study a random polynomial of degree $n$ over the finite field $\mathbb{F}_q$, where the coefficients are independent and identically distributed and uniformly chosen from the squares in $\mathbb{F}_q$. Our main result demonstrates that the likelihood of such a polynomial being irreducible approaches $1/n + O(q^{-1/2})$ as the field size $q$ grows infinitely large. The analysis we employ also applies to polynomials with coefficients selected from other specific sets.
\end{abstract}

\maketitle

\listoffixmes

\section{Introduction}
\label{sec:introduction}

Consider the random polynomial 
\begin{equation}
  \label{eq:general-random-polynomial}
  P_{\vecrv}\pa*{x} = x^n + \rv_1 x^{n-1} + \rv_{2} x^{n-2} + \dots + \rv_n,
\end{equation}
where the coefficients $\rv_1, \dots, \rv_n$ are independent random variables taking values in a ring $R$. 
The main question of interest is what is the probability that $P_{\vecrv}\pa*{x}$ is irreducible over $R$. 

When \( R = \ZZ \), one expects the probability to be close to $1$. To clarify this, we focus on two central models. The first, often called the large box model, assumes that the degree \( n \) is fixed and the coefficients \(\rv_i \) are uniformly distributed within the interval \( \{-H, \dots, H\} \) as \( H \to \infty \). In this case, not only is \( P_{\vecrv}(x) \) asymptotically almost surely (a.a.s.) irreducible, but the Galois group is also the full symmetric group a.a.s. The main challenge in this model is to bound the error term and demonstrate that it arises from reducible polynomials (see \cite{%
vanderwaerden1936seltenheit, %
knobloch1957hilbertschen, %
chela1963reducible, %
gallagher1973large, %
kuba2009distribution, %
zywina2010hilberts, %
dietmann2013probabilistic, %
chow2020enumerative, %
anderson2021quantitative, %
chow2023enumerative, %
chow2023towards}). This problem remains open despite a recent major breakthrough \cite{bhargava2021galois}.

In the second model, while the coefficients do not grow, the degree does—for example, when \( \rv_i \) are Rademacher variables. Here, we always condition on \( \xi_n \neq 0 \). In this context, \cite{breuillard2019irreducibility} proved that under the General Riemann Hypothesis (GRH), the polynomial is irreducible a.a.s. Unconditional results \cite{barysoroker2020irreducible, barysoroker2023irreducibility} show that the probability is bounded away from zero in full generality and under mild restrictions tend to $1$ (e.g., \( \rv_i \) are uniformly distributed on an interval of length \( \geq 35 \)). See \cite{barysoroker2024full} for a treatment of a hybrid model in which \( \rv_i \) are as in the large box model, but \( n \) may grow arbitrarily fast as a function of \( H \).

In this work we assume that  $R$ is a finite field. Let $q$ be a prime power and $R= \FF_q$  the finite field with $q$ elements. In this setting, the ring of polynomials $\FF_q[x]$ is closely analogous to the ring of integers $\mathbb{Z}$, and therefore the irreducible polynomials are analogous to prime numbers. And indeed, the first result is the prime polynomial theorem \cite[Theorem~2.2]{rosen2002number} which may be interpreted as the statement that if the $\rv_i$ are uniform on $\FF_q$, then
\begin{equation}
    \label{eq:PPT}
  \Pr\br*{P_{\vecrv}\pa*{x} \text{ irreducible over } \FF_q} = \frac{1}{n} + O\pa*{q^{-n/2}},
\end{equation}
as $q^n\to \infty$. There are variants when we fix  some of the coefficients, see \cite{pollack2013irreducible,ha2016irreducible}.
Motivated by Maynard's theorem~\cite{maynard2019primes} on primes without the digit 7, 
Porritt~\cite{porritt2019irreducible} 
considered polynomials with coefficients that are uniformly distributed in a subset  $S \subsetneq \FF_q$.
He proved that if $\#S \ge q - \sqrt{q}/2$ then
\begin{equation}
  \Pr\br*{P_{\vecrv}\pa*{x} \text{ irreducible over } \FF_q} \sim \frac{q}{n\pa*{q-1}},
\end{equation}
as $n\to \infty$. 
See \cite{moses2017irreducible} for generalizations.

In the case that $n$ is fixed, there are much stronger results. In \cite{cohen1972uniform, bank2015prime}, one can fix all coefficients but three, under very mild assumptions. A general criterion  by Entin~\cite{entin2021factorization} deals with uniform coefficients in a subset $S \subseteq \FF_q^n$ that is ``regular'' in some precise sense:
\begin{equation}
  \label{eq:entin-factorization}
  \Pr\br*{P_{\vecrv}\pa*{x} \text{ irreducible over } \FF_q} = \frac{1}{n} + O_n\pa*{q^{-1/2} \cdot \operatorname{irreg}\pa*{S}},
\end{equation}
where $\irreg\pa*{S}$ is an explicit constant that depends on $S$, see \eqref{eq:irregularity-of-set}.
This asymptotic formula is useful in several classical settings, such as when $S$ is a product of arithmetic progressions.

However, when $q$ is odd and $S = \set*{\alpha^2 \cond \alpha \in \FF_q}^n$ is the set of vector with square entries, the error term in \eqref{eq:entin-factorization} is bigger than the main term: $\irreg\pa*{S} \ge \pa*{q^{1/2}-1}^n$, see \eqref{thm:irregularity-of-squares}. Thus \eqref{eq:entin-factorization} gives no non-trivial information.
Our main result says that the probability for irreducibility is still asymptotically $1/n$:
\begin{theorem}
  \label{thm:main-squares-coefficients}
  Let $P_{\vecrv} \in \FF_q\br*{x}$ be a random polynomial as in \eqref{eq:general-random-polynomial} with $\rv_1, \dots, \rv_n$ uniformly distributed in $\set*{\alpha^2 \cond \alpha \in \FF_q} \subseteq \FF_q$.
  Then
  \begin{equation}
    \Pr\br*{P_{\vecrv}\pa*{x} \text{ irreducible over } \FF_q} = \frac{1}{n} + O_n\pa*{q^{-1/2}},
  \end{equation}
  as $q \to \infty$.
\end{theorem}
If $q$ is even, then  $\xi_i$ are uniform in $\FF_q$, hence it follows from \eqref{eq:PPT}. So from now on we assume $q$ is odd.

\autoref{thm:main-squares-coefficients} follows from a more general result,
in which the coefficients distribute uniformly in subsets of the following form.
For $M > 0$, we say that a subset $U \subseteq \FF_q$ is of  \emph{complexity $\leq M$} if there exist
\begin{equation}
  f_{11}, \dots f_{1k_1}, \dots, f_{m1}, \dots, f_{mk_m} \in \FF_q\br*{x}
\end{equation}
of positive degree such that 
\begin{equation}\label{eq:complexityM}
  U = \bigcup_{i=1}^m \bigcap_{j=1}^{k_i}f_{ij}(\FF_q)
  \qquad \text{and}\qquad
  \max_{\substack{i=1,\ldots, m\\j=1,\ldots,k_i}}\{m,k_i,\deg f_{ij} \} \leq  M
  .
\end{equation}
For example, the set of squares is of complexity $\leq 2$. Note that every subset $U$ of $\FF_ q$ is of complexity $\leq q$ since
\begin{equation}
  U = \bigcup_{\alpha \in U} f_\alpha\pa*{\FF_q} \qquad \text{where}\qquad 
  f_\alpha\pa*{x} = x^q - x + \alpha.
\end{equation}

\begin{theorem}
  \label{thm:main-general-coefficients}
  Let $q$ be an odd prime power, and  $U_1, \dots, U_n \subseteq \FF_q$ nonempty subsets  of complexity $\leq M$.
  Let $P_{\vecrv} \in \FF_q\br*{x}$ be a random polynomial as in \eqref{eq:general-random-polynomial} with $\rv_1, \dots, \rv_n$ uniformly distributed in $U_1, \dots, U_n$ respectively.
  Then
  \begin{equation}
    \Pr\br*{P_{\vecrv}\pa*{x} \text{ irreducible over } \FF_q} = \frac{1}{n} + O_{n,M}\pa*{q^{n-1/2} \cdot \#U_1^{-1} \cdots \#U_n^{-1}},
  \end{equation}
  as $q \to \infty$.
\end{theorem}

The error term goes to zero if $\# U_i \gg_M q$ for all $i$ (in the sense, $\# U_i\geq C_M q$). This happens in many cases, for example, if $U_i=f_i(\FF_q)$, for $1\leq \deg f_i\leq M$. 
In some cases, e.g.\ $U = \{ \alpha^2\} \cap \{-\alpha^2\}$ and $q\equiv 3\mod 4$ we have $\#U \ll_M 1$. We discuss this further in \autoref{sec:size-of-complex-sets}.


\subsection*{Acknowledgements}
The authors thank Alexei Entin for helpful discussions. 

This research was supported by the Israel Science Foundation (grant no.\ 366/23).


\section{Preliminaries}
\label{sec:preliminaries}
This section establishes the notation and revisits foundational facts necessary for the principal results.

\subsection{Splitting types}
Let $\Scl_n$ denote the set of $n$-tuples $\vecs = \pa*{s_1, \dots, s_n}$ of non-negative integers $s_i \geq 0$ with the constraint $s_1 + 2s_2 + \dots + ns_n = n$. The \emph{splitting type} of a permutation $\sigma$ in the symmetric group $\Sym$ on $n$ letters is defined as $\SplitType\pa*{\sigma}:=\vecs \in \Scl_n$ such that $s_i$ represents the number of cycles of length $i$ in the disjoint cycle decomposition of $\sigma$. Therefore, two permutations in $\Sym$ are conjugate if and only if they share the same splitting type. For $\vecs \in \Scl_n$, we denote by $C_{\vecs}$ the conjugacy class in $\Sym$ comprising all permutations with splitting type $\vecs$. The size of $C_{\vecs}$ is determined by Cauchy's formula:
\begin{equation}
  \label{eq:conjugacy-class-size}
  \#C_{\vecs} = n! \pa*{\prod_{i=1}^n i^{s_i} \cdot s_i!}^{-1}.
\end{equation}
Let $k$ be a field, and let $P \in k\br*{x}$ be a polynomial of degree $n$. The \emph{splitting type} of $P$ is $\SplitType\pa*{P}:= \vecs \in \Scl_n$ where $s_i$ is the number of irreducible factors of $P$ of degree $i$, counted with multiplicity.

\subsection{Discriminant}
Let $k$ be a field and let $P \in k\br*{x}$ be a separable polynomial.
Denote by $k_P$ the splitting field of $P$ over $k$ and set $G = \Gal\pa*{k_P / k}$.
Let $z_1, \dots, z_n$ be the roots of $P$ in $k_P$, and identify $G$ with a subgroup of $\Sym$ by the action of $G$ on the roots of $P$.

The discriminant of $P$ is defined as
\begin{equation}
  \disc P = \prod_{1 \le i < j \le n} \pa*{z_i - z_j}^2 \in k^\times.
\end{equation}
For $\sigma \in G$, we have that $\sigma\pa*{\sqrt{\disc P}} = \sign \pa*{\sigma} \sqrt{\disc P}$ where $\sign \pa*{\sigma}$ is the sign of the permutation $\sigma$. 
Thus, the following statement is immediate from the fundamental theorem of Galois theory:
\begin{proposition}
  \label{thm:odd-permutation-in-galois-with-char-not-two}
  Let $k$ be a field of characteristic $\neq 2$. 
  Then $G\not \leq \Alt$ if and only if $\disc P$ is not a square in $k$.
\end{proposition}

\subsection{Chebotarev's Density Theorem}
\label{sec:chebotarev-density-theorem}
We briefly introduce an explicit and uniform version of Chebotarev's Density Theorem over finite fields. All proofs can be found in \cite[Appendix A]{andrade2015shifted} in the language of rings or in  \cite[\S 4]{entin2019monodromy} in the language of schemes.

Let $F$  be a field which is a regular extension of the finite field $\FF_q$. Let $\Ocl_F$ be an integrally closed finitely generated $\FF_q$-algebra with fraction field $F$. Let $K / F$ be a finite Galois extension with Galois group $G$, let $\Ocl_K$ be the integral closure of $\Ocl_F$ in $K$, and let $\FF_{q^r}$ be the relative algebraic closure of $
\FF_q$ in $K$.
In particular, $\FF_{q^r}$ is contained in $\Ocl_K$.
For each $v \ge 0$, we define
\begin{equation}
  \label{eq:galois-v}
  G_v = \set*{\sigma \in G \cond \sigma\pa*{\alpha} = \alpha^{q^v} \text{ for all } \alpha \in \FF_{q^r}}.
\end{equation}
Then $G_0\leq G$ is the kernel of the restriction map $G\to \mathbb{F}_{q^r}$. If $v\geq 1$, then  $G_v$ is a set on which $G_0$ acts by conjugation. Let $\Ccl\pa*{G_v}$ denote the set of orbits. 

\begin{remark}
    \label{rem:Gvforregular}
    If $K$ is regular over $\FF_q$, then $r=1$, hence $G_v=G$ for all $v\geq 0$. 
\end{remark}

Let $\Phi \in \Hom_{\FF_{q^r}}\pa*{\Ocl_K, \bar \FF_q}$ be unramified over $\Ocl_F$ (in the sense that $\ker\Phi$ is unramified).
Then, there exists a unique element in $G$, which we call the Frobenius element, and we denote by
\begin{equation}
  \frob{\Phi} \in G,
\end{equation}
such that for all $u \in \Ocl_F$
\begin{equation}
\label{eq:actionoffrob}
  \Phi\pa*{\frob{\Phi} u} = \Phi\pa*{u}^q.
\end{equation}

If we fix a surjective $\phi \in \Hom_{\FF_q}\pa*{\Ocl_F, \FF_{q^v}}$ that is unramified 
in $\Ocl_K$, then the set 
\begin{equation}
  \frobC{\phi} = \set*{\frob{\Phi} \cond \Phi \in \Hom_{\FF_q}\pa*{\Ocl_K, \bar \FF_q} \text{ extends } \phi} \subseteq G_v,
\end{equation}
is invariant under conjugation from $G_0$, since ${\frob{\Phi}}^\tau= \frob{\Phi\circ \tau}$, $\tau\in G_0$. Moreover, $G_0$ acts transitively on $\frobC{\phi}$, so  $\frobC{\phi} \in  \Ccl\pa*{G_v}$.

We use a complexity notion $\cmp(\Ocl_F)=\cmp(\Ocl_F/\Ocl_F)$ and $\cmp(\Ocl_K/\Ocl_F)$ as defined in   \cite[Section 4]{entin2019monodromy} in the language of schemes.
In particular, the complexity of 
\begin{equation}
  \label{eq:fq-algebra-general-form}
  \FF_q\br*{x_1, \dots, x_n, f_0^{-1}} / \pa*{f_1(x_1,\ldots, x_n), \dots, f_m(x_1,\ldots, x_n)},
\end{equation}
is bounded by $\max(n, \deg f_0,\ldots, \deg f_m)$, and the complexity of $S/R$ is a bounded by a function of  $\cmp\pa*{R}$, $\cmp\pa*{S}$, and the degree of the extension $S / R$.

By the Lang-Weil's estimates \cite{lang1954number}, the number of ramified $\phi$ is $O_{\cmp(\Ocl_K/\Ocl_F)}(q^{n-1})$. We shall use this bound in the rest of the work multiple times. 

In this language, Chebotarev's theorem gives a quantitative equidsitribution of $\frobC{\phi}$ in $\Ccl\pa*{G_v}$: 
For $C \in \Ccl\pa*{G_v}$, we define
\begin{equation}
  \dchebotarv[v]\pa*{C} =\frac{1}{q^{vd}} \cdot \# \set*{\phi \in \Hom_{\FF_q}\pa*{\Ocl_F, \FF_{q^v}} \cond \begin{array}{c}\phi \text{ surjective, unramified in } \Ocl_K \\ \text{and } \frobC{\phi} = C\end{array}}.
\end{equation}
By the Lang-Weil estimates \cite{lang1954number}, the number of surjective $\phi\in \Hom_{\FF_q}\pa*{\Ocl_F,\FF_{q^v}}$ is equal to $q^{vd}+O(q^{v(d-1)})$, so $\dchebotarv[v]\pa*{C}$ is the asymptotic density of such $\phi$-s with $\frobC{\phi} = C$.

\begin{theorem}[Chebotarev's Density Theorem]
  \label{thm:chebotarev-density-theorem}
  Let $v\geq 1$ and $C \in \Ccl\pa*{G_v} $.
  Then
  \begin{equation}
    \label{eq:chebotarev-density-theorem}
    \dchebotarv[v]\pa*{C} = \frac{\#C}{\#G_v} + O_{\cmp\pa*{\Ocl_K / \Ocl_F}}\pa*{q^{-1/2}}.
  \end{equation}
\end{theorem}
\begin{proof}
    See \cite[Theorem A.4]{andrade2015shifted}\footnote{The implied constant in \emph{loc.cit.} is written to depend only on $\cmp(\Ocl_F)$ and $[K:F]$. This is not true, and the correct error term is as written here.} and \cite[Theorem 3]{entin2019monodromy}.
\end{proof}

\begin{remark}
  In \cite{andrade2015shifted} the authors require that $\Ocl_K / \Ocl_F$ is unramified.
  This requirement may be added by localizing at the ramified locus,  and by the Lang-Weil estimates \cite{lang1954number} applied to the discriminant of $\Ocl_K / \Ocl_F$. This will change the formula by $O_{\cmp\pa*{\Ocl_K/\Ocl_F}}\pa*{q^{-1}}$, hence will not affect  \eqref{eq:chebotarev-density-theorem}.
\end{remark}

\begin{remark}
    Obviously, \autoref{thm:chebotarev-density-theorem} generalizes to any $G_0$-invariant subset of $C \subseteq G_{v}$. Indeed, such $C$ is the disjoint union of elements in $\Ccl(G_v)$. Note if $C =\varnothing$, then $\dchebotarv[v](C)=0$.
\end{remark}

    We state and prove a well-known corollary that will be used in the sequel. 

    \begin{corollary}
    \label{thm:chebotarev_lin_dis}
        In the notation above, assume that $K=LM$ for $L,M$ linearly disjoint Galois extensions of $F$, and $L$ regular over $\FF_q$. Let $A=\Gal(L/F)$, $B=\Gal(M/F)$, let $v\geq 0$, let $C\subseteq A$ be a conjugacy class,  and $D\in \Ccl(B_{v})$. Then, $C\times D\in \Ccl(G_v)$ and 
        \[
            \dchebotarv[v]\pa*{C\times D} = \dchebotarv\pa*{C}\cdot \dchebotarv[v]\pa*{D}+O_{\cmp\pa*{\Ocl_K/\Ocl_F}}\pa*{q^{-1/2}}.
        \]
    \end{corollary}

\begin{proof}
    We may identify $G\cong A\times B$ via the map $\sigma \mapsto (\sigma|_{L},\sigma|_{M})$. 
    Since $L$ is regular over $\FF_q$
    and since $L,M$ are linearly disjoint over $F$, it follows that the algebraic closures of $\FF_q$ in $M$ and in $K$ coincide. Hence, $G_v = A\times B_v$ and $C\times D\in \Ccl(G_v)$.
    Finally, if $\phi\in \Hom_{\FF_q}\pa*{F,\FF_{q^v}}$ is surjective and unramified in $K$, then by \eqref{eq:actionoffrob} one may readily deduce that 
    \[
        \frobC{\phi} = \pa*{\frobC[\Ocl_{L}/\Ocl_F]{\phi},\frobC[\Ocl_{M}/\Ocl_F]{\phi}}.
    \]
    The assertion follows from \autoref{thm:chebotarev-density-theorem} applied to each of the terms, and noting that $\cmp\pa*{\Ocl_{L}/\Ocl_F}$, $\cmp\pa*{\Ocl_{M}/\Ocl_F}$ are bounded in terms of $\cmp\pa*{\Ocl_{K}/\Ocl_F}$ and that the number of $\phi\in \Hom_{\FF_q}\pa*{F,\FF_{q^v}}$ that are ramified in any of the fields $K,L,M$ is $O_{\cmp\pa*{\Ocl_{K}/\Ocl_F}}\pa*{q^{-1}}$.
\end{proof}

\subsection{Two examples}
We explicitly  compute two examples that will be used in the proof of the main result. 

\begin{example}
    \label{ex:generic}
Let $\vect=\pa*{t_1, \dots, t_n}$ be an $n$-tuple of indeterminates. The \emph{generic polynomial of degree $n$} is defined to be 
\begin{equation}
  \label{eq:generic-polynomial}
  P\pa*{\vect; x} = x^n + t_1 x^{n-1} + \dots + t_n \in \FF_q\pa*{\vect}\br*{x}. 
\end{equation}
Let $F=\FF_q(\vect)$, let $F_P$ be the splitting field of $P$ over $F$. Then $G=\Gal(F_P/F)\cong \Sym$ and $F_P/\FF_q$ is regular. Moreover, $\phi\in \Hom_{\FF_q}(F,\FF_q)$ is unramified in $F_P$ if and only if $P^{\phi}:= P\pa*{\phi{\pa*\vect};x}\in \FF_q[x]$ is separable. 

From \eqref{eq:actionoffrob} it follows that if $\phi\in \Hom_{\FF_q}\pa*{F,\FF_q}$ is unramified in $F_P$, then 
\begin{equation}
    \label{eq:actionfrobgenericpolynomial}
    \SplitType\pa*{P^{\phi}}=\SplitType\pa*{\frobC[\Ocl_{F_P}/\Ocl_{F}]{\phi}}.
\end{equation}
Since $\cmp(\Ocl_{F_P}/\Ocl_{F})$ is bounded by a function in $n$ and since the number of ramified $\phi$ is $O_n(q^{-1})$,  \autoref{thm:chebotarev-density-theorem} and \eqref{eq:conjugacy-class-size} imply that for any $\vecs\in \Scl_n$ we have
\begin{equation}\label{eq:computation_generic}
    \begin{split}       
    &\frac{1}{q^n} \set*{\phi\in  \Hom_{\FF_q}(F,\FF_q) \cond \SplitType\pa*{P^{\phi}} = \vecs} \\
    &\qquad=\frac{1}{q^n} \set*{\phi\in  \Hom_{\FF_q}(F,\FF_q) \cond 
    \phi\text{ ramified or }\SplitType\pa*{\frobC[\Ocl_{F_P}/\Ocl_{F}]{\phi}} = \vecs}
    \\ &\qquad=\pa*{\prod_{i=1}^n i^{s_i} \cdot s_i!}^{-1} +O_n(q^{-1/2}).
    \end{split}
\end{equation}
\end{example}

\begin{example}
    \label{ex:imagesoffij}
Consider univariate polynomials $f_{ij} (x)\in \FF_q[x]$ of degree $>0$,  for $i=1,\ldots, n$ and $j=1,\ldots, k_i$. Define 
\begin{equation}
    \label{eq:defSi}
    S_i = \bigcap_{j=1}^{k_i} f_{ij}(\FF_q)\subseteq \FF_q.
\end{equation}
We may assume without loss of generality that $f_{ij}(x)-t_i$ is separable in $x$. If this is not the case, $f_{ij}(x)=\tilde{f}_{ij}(x^{p^{r}})$ for some $r>0$ and $\tilde{f}_{ij}(x)-t_i$ separable. Moreover, $f_{ij}(\FF_q)=\tilde{f}_{ij}(\FF_q)$, so we may replace $f_{ij}$ by $\tilde{f}_{ij}$ without altering the sets $S_i$. 

Let $K_{ij}$ be the extension of $\FF_q(t_i)$ generated by a root of $f_{ij}(x)-t_i$, let $N_{ij}$ the Galois closure of $K_{ij}$ over $\FF_{q}(t_i)$, $N_i=\prod_{j=1}^{k_i} N_{ij}$, and $N=N_1\cdots N_n$. 
Then $N/F$ is Galois. Let $G_N=\Gal(N/F)$, $G_{ij}=\Gal(N_{ij}/\FF_{q}(t_i))$, $H_{ij}=\Gal(N_{ij}/K_{ij})$, and $\pi_{i,j}\colon G_N\to G_{ij}$ the restriction map. 

Let $\phi\in \Hom_{\FF_q}(\Ocl_{F},\FF_q)$ be unramified in $\Ocl_N$. 
Then, $\phi_{i}=\phi|_{\FF_q(t_i)}$ is unramified in $N_{ij}$ for all $i,j$. 
We have that $\phi(t_i)=\phi_i(t_i)\in f_{ij}\pa*{\FF_q}$ if and only if $\phi_i$ extends to $\Phi_{ij}\colon \Ocl_{N_{ij}} \to \bar \FF_q$ 
with $\Phi_{ij}(\Ocl_{K_{ij}})=\FF_q$ if and only if $\phi_i$ extends to $\Phi_{ij}\colon \Ocl_{N_{ij}} \to \bar \FF_q$  such that $\frob[\Ocl_{N_{ij}}/\FF_{q}{[t_i]}]{\Phi_{ij}} \in H_{ij}$ if and only if $\frobC{\phi_i} \cap H_{ij} \ne \varnothing$. 
In other words, $\phi(t_i)\in f_{ij}(\FF_q) $ if and only if $\frobC[\Ocl_{N}/\Ocl_F]{\phi} \cap \pi_{ij}^{-1}(H_{ij})\neq \varnothing$. Therefore,
\begin{equation}
    \label{eq:whenphitiinSi}
    \phi(\vect)\in S_1\times\cdots \times S_n \quad \Longleftrightarrow \quad \frobC[\Ocl_{N}/\Ocl_{F}]{\phi} \cap \pi_{ij}^{-1} (H_{ij})\neq \varnothing, \quad \forall i,j.
\end{equation}
Now by \autoref{thm:chebotarev-density-theorem} applied to the $G_{N,0}$-invariant set 
\begin{equation}
\Omega = \{ \sigma\in G_{N,1} : \forall i, j, \; \exists \sigma_0\in G_{N,0} \mbox{ such that } \sigma_0\sigma \sigma_0^{-1} \in \pi^{-1}_{ij}(H_{ij})\} 
\end{equation}
we get that 
\begin{equation}\label{imageofmaps}
    \dchebotarv[1](\Omega) = \frac{\#\Omega}{\#G_{N,1}} +O_{M}(q^{-1/2}),
\end{equation}
where $M= \max_{i,j}\{\deg f_{ij}\}$ and $\cmp(\Ocl_N/\Ocl_F)$ is bounded in terms of $M$.
We finish this example by adding the ramified homomorphisms. 
\begin{equation}
\label{eq:comparisonbetweendchebotarevandsizeofSi}
\begin{split}
    \left|\frac{1}{q^n}\prod_{i=1}^n \#S_i -\dchebotarv[1](\Omega)\right|
    &\leq \frac{1}{q^n}\#\set*{\phi\in \Hom_{\FF_q}(\Ocl_{F},\FF_q) \cond
    \begin{array}{c}
      \phi \text{ ramified in } N \text{ and } \\
      \phi\pa*{\vect} \in S_1\times \cdots \times S_n
    \end{array}
    } \\ 
    & =  O_M(q^{-1}).
\end{split}
\end{equation}
\end{example} 

\section{General Theorem}
\label{sec:general-theorem}
Our most general theorem is stated below. We begin by showing how \autoref{thm:main-squares-coefficients} and \autoref{thm:main-general-coefficients} follow from it. The remainder of the paper focuses on proving this theorem.
\begin{theorem}
  \label{thm:main}
  Let $q$ be an odd prime power, $\vecs\in \Scl_n$, and let $U_1, \dots, U_n \subseteq \FF_q$ be subsets of complexity $\leq M$.
  Let $P_{\vecrv} \in \FF_q\br*{x}$ be a random polynomial as in \eqref{eq:general-random-polynomial} where $\rv_1, \dots, \rv_n$ are  uniformly distributed in $U_1, \dots, U_n$ respectively.
  Then
  \begin{equation}
    \Pr\br*{\SplitType\pa*{P_{\vecrv}} = \vecs} = \pa*{\prod_{i=1}^n i^{s_i} \cdot s_i!}^{-1} + O_{n,M}\pa*{q^{n-1/2} \cdot \#U_1^{-1} \cdots \#U_n^{-1}},
  \end{equation}
  as $q \to \infty$.
\end{theorem}

\begin{proof}[Proof of \autoref{thm:main-general-coefficients}]
    As $P_{\vecrv}$ is irreducible if and only if $\SplitType\pa*{P_{\vecrv}} = \pa*{0, \dots, 0, 1}$, and in this case $\pa*{\prod_{i=1}^n i^{s_i} \cdot s_i!}^{-1}=1/n$, we see that \autoref{thm:main-general-coefficients} is a special case of \autoref{thm:main}.
\end{proof}

\begin{proof}[Proof of \autoref{thm:main-squares-coefficients}]
  As mentioned above, if $q$ is even, then the coefficients are uniform, hence we are done by \eqref{eq:PPT}.
  If $q$ is odd, then \autoref{thm:main-squares-coefficients} follows from \autoref{thm:main-general-coefficients} applied to the sets $U_1 = \dots = U_n = \set*{\alpha^2 \cond \alpha \in \FF_q}$ which have complexity $\le 2$. (Note that  $\#U_i = \pa*{q + 1}/2$ for all $i$, hence the error term is $O_n(q^{-1/2})$.)
\end{proof}

\section{Linearly Disjoint Extensions}
The goal of this section is \autoref{thm:linearly-disjoint-extensions} which asserts that the splitting field of the generic polynomial is linearly disjoint of any extension that decomposes to extensions of one variable. This is a technical but crucial point in our proof of \autoref{thm:main}.

\begin{definition}
  \label{def:decomposable-extension}
  An extension $N / k\pa*{\vect}$ is called \emph{univariate-decomposable} or simply \emph{decomposable} if there exist finite extensions $N_i / k\pa*{t_i}$, $i=1,\dots,n$, such that $N = N_1 \cdots N_n$.
\end{definition}

\begin{proposition}
  \label{thm:linearly-disjoint-extensions}
  Let $q$ be an odd prime power. Let $P(\vect;x) =x^{n}+t_1x^{n-1}+\cdots +t_n$ be the generic polynomial of degree $n$ over $\FF_q$, let $F_P$ be its splitting field over $\FF_q\pa*{\vect}$, and let $N / \FF_q\pa*{\vect}$ be a decomposable extension.
  Then $F_P$ and $N$ are linearly disjoint over $\FF_q\pa*{\vect}$.
\end{proposition}

We start with some auxiliary lemmas, keeping the notation and assumptions of the proposition. We also set
\begin{equation}
  \vect_{\hat i} = \pa*{t_1, \dots, t_{i-1}, t_{i+1}, \dots, t_n} 
\end{equation}
and 
we let $L=N\bar\FF_q$ and $L_i=N_i\bar\FF_q(t_i)$, where $\bar\FF_q$ is the algebraic closure of $\FF_q$. Then, $L/\bar\FF_q(\vect)$ is decomposable, since  $L=L_1\cdots L_n$.

\begin{lemma}
  \label{thm:no-square}
  Let $L_u$ be a finite extension of $\bar \FF_q\pa*{u}$.
  Then there exists $\alpha \in \bar\FF_q$ such that $u - \alpha$ is not a square in $L_u$.
\end{lemma}

\begin{proof}
  As $E=\bar\FF_q(u)\pa*{\sqrt{u-\alpha}\cond \alpha\in \bar\FF_q}$ is an infinite extension of $\bar\FF_q(u)$, and $L_u$ is a finite extension, it follows that $E\not\subseteq L_u$. In particular, there exists $\alpha\in \bar\FF_q(u)$ such that $\sqrt{u-\alpha} \not\in L_u$.  
\end{proof}

\begin{lemma}
  \label{thm:no-square-root-of-disc}
  Assume $n\geq 2$. 
  Then $\disc P$ is not a square in $L$.
\end{lemma}

\begin{proof}
  Recall that  $L = L_1 \cdots L_n$, where  $L_i / \bar\FF_q\pa*{t_i}$ are  finite extensions, $i=1,\dots,n$.
  For each $i=1,\dots,n$, let $\Ocl_{L_i}$ be the integral closure of $\bar\FF_q\br*{t_i}$ in $L_i$ and let $\Ocl_L$ be the integral closure of $\bar\FF_q\br*{\vect}$ in $L$.

  We proceed by induction on $n\geq 2$.
  Assume $n=2$. Then we have that
  \begin{equation}
    \disc P = \disc (x^2+t_1x+t_2) = t_1^2 - 4 t_2.
  \end{equation}
  Since $L_2 / \bar \FF_q\pa*{t_2}$ is finite, by \autoref{thm:no-square} there exists $\alpha \in \bar \FF_q$ such that $t_2 - \alpha$ is not a sqaure in $L_2$.
  Extend the specialization $(t_1,t_2)\mapsto (2\sqrt{\alpha},t_2)$ to a homomorphism $\varphi \colon \Ocl_L \to L_2$. 
 Then, $\phi(\disc P) = -4\pa*{t_2 - \alpha}$ which is not a square in $L_2$. This implies that $\disc P$ is not a square in $L$, as needed. 

  Next let $n > 2$. Consider the  polynomials
  \begin{equation}
    \label{eq:sp:hat-p-and-q}
    \begin{aligned}
      \hat P \pa*{\vect_{\hat n}; x}
        &= P\pa*{\vect_{\hat n}, 0; x}
        = x^n + t_1 x^{n-1} + \dots + t_{n-1}x, \quad\text{and} \\
      Q \pa*{\vect_{\hat n}; x}
        &= \hat P\pa*{\vect_{\hat n}; x} / x = x^{n-1} + t_1 x^{n-2} + \dots + t_{n-1}.
    \end{aligned}
  \end{equation}
  Then, by the definition of discriminants and resultants we have
  \begin{equation}
      \disc \hat P = \disc \pa*{xQ} = \Res\pa*{x, Q}^2 \disc Q = t_{n-1}^2 \disc Q.
  \end{equation}
  Set $L' = L_1 \cdots L_{n-1}$.
  The induction hypothesis for $n-1$ implies that $\disc Q$ is not a square in $ L'$, hence $\disc \hat P$ is not a square in $L'$. Since $\disc \hat P$ is obtained by substituting $t_n=0$ in $\disc P$, we deduce that  $\disc P$ is not a square in $ L$.
\end{proof}

\begin{lemma}
  \label{thm:geometric-galois-group}
  The Galois group of $P$ over $L$ is $\Sym$.
\end{lemma}
\begin{proof}
    Let $F = \bar \FF_q\pa*{\vect}$. Let $\hat{L}$ be the normal closure of $L/F$, then $\hat{L}$ is also decomposable (since if $L=L_1\cdots L_n$, then  $\hat{L} = \hat{L}_1\cdots \hat{L}_n$), with $\hat{L}_i$ the normal closure of $L_1/\FF_q\pa*{t_i}$. As the Galois group $\hat{G}$ of $P$ over $\hat{L}$ is a subgroup of its Galois group over $L$, it suffices to prove that $\hat{G}=\Sym$. Hence, without loss of generality, we may assume that $L/F$ is normal.
    
  Let $F_P$ and $L_P$ be the splitting fields of $P$ over $F$ and $L$, respectively.   We set $G_{F} = \Gal\pa*{F_P / F}$ and $G_L = \Gal\pa*{L_P / L}$. Then $G_F\cong \Sym$ since $P$ is the generic polynomial. 
  
  Since $F_P/F$ and $L/F$ are normal, so is $E = F_P \cap L$. Identifying $G_L=\Gal(F_P/E)$ via the restriction map, gives that $G_L$ is a normal subgroup of $G_F=\Sym$. 

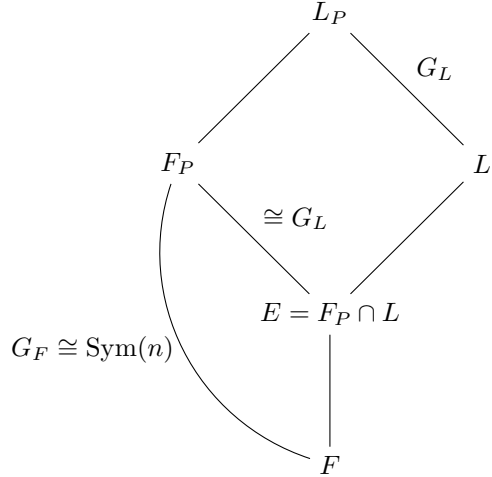
  \begin{figure}[htb]
    \centering
    \begin{tikzpicture}
      \node (LP) at (0, 0) {$L_P$};
      \node (KP) at (-2, -2) {$F_P$};
      \node (L) at (2, -2) {$L$};
      \node (E) at (0, -4) {$E = F_P \cap L$};
      \node (K) at (0, -6) {$F$};
      \draw (KP) -- (LP);
      \draw (L) -- (LP) node [midway, above right] {$G_L$};
      \draw (E) -- (KP) node [midway, above right] {$\cong G_L$};
      \draw (E) -- (L);
      \draw [bend left=45] (K) to node [midway, left] {$G_F \cong \Sym$} (KP) ;
      \draw (K) -- (E);
    \end{tikzpicture}
    \label{fig:ggps:field-extensions}
  \end{figure}
  
  Since all the proper normal subgroups of $\Sym$ are contained in $\Alt$, it suffices to prove that $G_L\not\leq \Alt$ to deduce that $G_L=\Sym$. 
  And indeed,  \autoref{thm:no-square-root-of-disc} and \autoref{thm:odd-permutation-in-galois-with-char-not-two} imply that $G_L\not\leq \Alt$, so $G_L \cong \Sym$.
\end{proof}

\begin{proof}[Proof of \autoref{thm:linearly-disjoint-extensions}]
  Set $K = \FF_q\pa*{\vect}$.
  Let $K_P$, $N_P$ and $L_P$ be the splitting fields of $P$ over $K$, $N$, and $L$, respectively, and let $G_K, G_N, G_L$ be the Galois groups of $K_P/K$, $N_P/N$, and $L_P/L$. We identify these groups with their images under the restriction of automorphism map: $G_N\cong \Gal(K_P/N\cap K_P)$ and $G_L\cong \Gal(K_P/L\cap K_P)$, so that $G_L\leq G_N\leq G_K\cong \Sym$. By \autoref{thm:geometric-galois-group}, $G_L=\Sym$, hence $G_N=\Sym=G_K$. In particular, $N\cap K_P=K$, and so $N$ and $K_P$ are linearly disjoint.   
 \end{proof}

\section{Proof of \autoref{thm:main} }
\label{sec:main}
We first treat a special case, and then deduce the theorem.

\begin{definition}
  Let $M>0$. We say that  a subset $S \subseteq \FF_q$ is of \emph{$\pi$-complexity $\leq M$} if it has a representation as in \eqref{eq:complexityM} with $m=1$.
\end{definition}

There are some obvious connections between the $\pi$-complexity and complexity of sets. First, if the $\pi$-complexity is $\leq M$, then so is the complexity.
Second, if $U$ is of complexity $\leq M$ then it is a union of at most $M$ sets of $\pi$-complexity $\leq M$. Another immediate property that will be used later is that if $S_i$ is of $\pi$-complexity $\leq M_i$, $i=1,\ldots r$, then 
\begin{equation}\label{eq:picmpunderintersection}
    \bigcap_{i=1}^r S_i \text{ is of $\pi$-complexity }\leq \sum_{i=1}^r M_i.
\end{equation}

\begin{lemma}
  \label{thm:pi-polynomial-coeffients}
  Let $q$ be an odd prime power.
  Let $S_1, \dots, S_n \subseteq \FF_q$ be nonempty and of $\pi$-complexity $\leq M$.
  Let $P_{\vecrv} \in \FF_q\br*{x}$ be a random polynomial as in \eqref{eq:general-random-polynomial} where $\rv_1, \dots, \rv_n$ are uniformly distributed in $S_1, \dots, S_n$ respectively.
  Then for all $\vecs \in \Scl_n$ we have that
  \begin{equation}\label{eq:bdinmainlemma}
    \Pr\br*{\SplitType\pa*{P_{\vecrv}} = \vecs} = \pa*{\prod_{i=1}^n i^{s_i} \cdot s_i!}^{-1} + O_{n,M}\pa*{q^{n-1/2} \cdot \#S_1^{-1} \cdots \#S_n^{-1}},
  \end{equation}
  as $q \to \infty$.
\end{lemma}

\begin{proof}
    Let $F=\FF_q(\vect)$ and $\Ocl_F=\FF_q[\vect]$. Let $P(\vect;x)=x^n+t_1x^{n-1}+\cdots +t_n$ be the generic polynomial, $F_P$ its splitting field over $F$, and $\Ocl_{F_P}$ the integral closure of $\Ocl_F$ in $F_P$. In particlar, $F_P/\FF_q$ is regular. 

    By assumption, there exist $f_{11}, \dots, f_{1k_1}, \dots, f_{n1}, \dots, f_{nk_n} \in \FF_q\br*{x}$  such that $k_i \leq  M$, $1\leq \deg f_{ij} \leq  M$ for all $i, j$ such that  
    \begin{equation}
        S_i = \bigcap_{j=1}^{k_i} f_{ij}(\FF_q), \qquad i=1,\ldots, n.
    \end{equation}
    As in \autoref{ex:imagesoffij} we may assume that 
    $f_{ij}(x)-t_i$ is separable in $x$. We also assume the notation of \autoref{ex:imagesoffij}. In particular, as $N=N_1\cdots N_n$, we have that $N$ is decomposable. 
    By \autoref{thm:linearly-disjoint-extensions}, $N$ and $F_P$ are linearly disjoint over $F$. Let $C_{\vecs}\subseteq \Gal(F_P/F)\cong \Sym$ be the conjugacy class of all permutations with splitting type $\vecs$ and let $\Omega\subseteq \Gal(N/F)_1$ be the $G_0$-invariant subset defined in \autoref{ex:imagesoffij}. Then, by \autoref{thm:chebotarev_lin_dis}, \eqref{eq:computation_generic}, and \eqref{imageofmaps}, we have
    \[
        \dchebotarv[1](C_{\vecs}\times \Omega) = \pa*{\prod_{i=1}^n i^{s_i} \cdot s_i!}^{-1} \cdot \frac{\#\Omega}{\#G_{N,1}} +O_{n,M}\pa*{q^{-1/2}}.
    \]
    By \eqref{eq:comparisonbetweendchebotarevandsizeofSi} and since $\pa*{\prod_{i=1}^n i^{s_i} \cdot s_i!}^{-1}\leq 1$, we get that
    \[
        \dchebotarv[1](C_{\vecs}\times \Omega) = \pa*{\prod_{i=1}^n i^{s_i} \cdot s_i!}^{-1} \cdot \frac{\#S_1\cdots \#S_n}{q^n} +O_{n,M}\pa*{q^{-1/2}}.
    \]
    We identify $\Hom_{\FF_q}(\FF_{q}[\vect],\FF_q) \cong \FF_q^n$ by the bijection $\phi\mapsto \veca:= \phi(\vect)$. For $\phi$ unramified in $F_PN$, We have   $\frobC[\Ocl_{F_pN}/\Ocl_{F}]{\phi} \in C_{\vecs}\times \Omega$ if and only if $\frobC[\Ocl_{F_p}/\Ocl_{F}]{\phi} = C_{\vecs}$ and $\frobC[\Ocl_{N}/\Ocl_{F}]{\phi}\in \Omega$. The former condition is equivalent to $\SplitType\pa*{P(\veca; x)} = s$ by \autoref{ex:generic} and the latter to $\veca\in S_1\times \cdots \times S_n$ by \autoref{ex:imagesoffij}. Since there are at most 
     $O_{n,M}(q^{n-1})$ many ramified $\phi$,  we deduce that  
    \begin{multline}
        \pa*{\prod_{i=1}^n i^{s_i} \cdot s_i!}^{-1} 
        = \frac{\dchebotarv[1](C_{\vecs}\times \Omega)}{\#S_1\cdots \#S_n/q^n} + O_{n,M}(q^{n-1/2}\#S_1^{-1}\cdots \#S_n^{-1})\\
        = \frac{\# \{ \veca \in \FF_q^n : \SplitType\pa*{P(\veca;x)}=\vecs \text{ and } \veca \in S_1\times \cdots \times S_n\}}{\#S_1\cdots \#S_n}\\
         +O_{n,M}(q^{n-1/2}\#S_1^{-1}\cdots \#S_n^{-1}).
    \end{multline}
    This finishes the proof as the coefficients $P_{\vecrv}$ are sampled uniformly from $S_1\times \cdots \times S_n$. 
\end{proof}

\begin{proof}[Proof of \autoref{thm:main}]
    For each $i=1,\ldots, n$, the set $U_i$ is of complexity $\leq M$. This means that there exists $m_i\leq M$ and there exist $S_{i1},\ldots, S_{im_i}\subseteq \FF_q$  of $\pi$-complexity $\leq M$ such that $U_i=\bigcup_{j=1}^{m_i}S_{ij}$. 

    If we set 
    \begin{equation}
        \Fcl = \set*{\hat{S}= S_{1j_1} \times \dots \times S_{nj_n} \cond 1 \le j_i \le m_i},
    \end{equation}
  then
  \[
    \hat{U}:= U_1\times \cdots \times U_n=\bigcup_{\hat{S}\in \Fcl} \hat{S}.
  \]
  Let $E$ be the event that $\SplitType(P_{\vecrv})=\vecs$. 
  Also, for simplicity we identify subsets $\hat S \subseteq \FF_q^n$ with the event that $\vecrv \in \hat S$. 
    By the inclusion-exclusion principle we get
  \begin{equation}
  \label{eq:incecl}
    \begin{split}
        \PP\br*{E} &= \sum_{i=1}^{\#\Fcl} (-1)^{i-1} \sum_{\substack{\hat{S}_1,\ldots, \hat{S}_i\in \Fcl\\ \text{distinct}}} \PP\br*{E\cap \bigcap_{k=1
        }^i \hat{S}_k} 
    \end{split}
    \end{equation}
    To ease notation,  we fix for a moment one tuple $(\hat{S}_1,\ldots, \hat{S}_i)$ of distinct elements of $\mathcal{F}$, and write  $\hat{S}=\bigcap_{k=1
    }^i \hat{S}_k$. Since $\hat{S}\subseteq \hat{U}$ and $\vecrv$ is uniform on $\hat{U}$,  we have that 
    \begin{equation}\label{eq:sizeofS}
    \PP\br*{\hat{S}} = \frac{\#\hat S}{\#\hat U}.
    \end{equation}
    Also, if we condition on $\vecrv$ to be in $\hat{S}$, then it distributes  uniformly on $\hat{S}$. Hence, we may apply \autoref{thm:pi-polynomial-coeffients}, noting that by  \eqref{eq:picmpunderintersection}, the $\pi$-complexity of each factor of $\hat{S}$ is $\leq \#\Fcl\cdot M=O_{M}( 1)$ and using \eqref{eq:sizeofS}:  \begin{equation}\label{eq:EcapS}
        \PP\br*{E\cap\hat{S}} =\PP\br*{E\cond\hat{S}}\PP\br*{\hat{S}}  
        = C\cdot \PP\br*{\hat{S}} +O_{n,M}(q^{-1/2}\#\hat{U}^{-1}) ) .
\end{equation}
    where $C= \pa*{\prod_{i=1}^n i^{s_i} \cdot s_i!}^{-1}$. 
    Plugging \eqref{eq:EcapS} into \eqref{eq:incecl} and noting that the sum in the latter has $O_{n,M}(1)$ terms, we deduce that 
    \[
        \PP\br*{E} = C\sum_{i=1}^{\#\Fcl}(-1)^{i-1} \sum_{\substack{\hat{S}_1,\ldots, \hat{S}_i\in \Fcl\\ \text{distinct}}}  \PP\br*{\bigcap_{k=1}^i \hat{S}_k} 
        + O_{n,M}\pa*{q^{-1/2}\#\hat{U}^{-1} }=C+O_{n,M}\pa*{q^{-1/2}\#\hat{U}^{-1}},
    \]
    where in the last equation we used the inclusion-exclusion principle for the event $\hat{U}$. 
\end{proof}

\section{Examples of large size sets of bounded complexity}
\label{sec:size-of-complex-sets}

In \autoref{thm:main-general-coefficients} and \autoref{thm:main}, the error term may be larger than the main term. This happens if  $\#U_i \ll_M q^{1/2}$ for at least one $i$. One the other hand, if $\#U_i\gg_M q$ for all $i$, then the error term is of size $O_{M}(q^{-1/2})$ and it tends to zero as $q\to \infty$. 

Let $U$ be a set of complexity $\leq M$. 
The goal of this section is to show that if $\#U\not\gg_M q$, then $\#U\ll_M 1$, and to give a sufficient condition for $\#U\gg_M q$. 

Since $U$ is a union of sets $S$ of $\pi$-complexity $\leq M$, we restrict the discussion to the latter sets.
\begin{proposition}\label{thm:dicotomy}
    Let $S\subseteq \FF_q$ be a set of $\pi$-complexity $\leq M$. Then either $\#S\ll_M1$ or $\#S \gg_M q$. 
\end{proposition}

\begin{proof}
    We use  \autoref{ex:imagesoffij} and its notation, taking $n=1$, $t=t_1=\vect$, $\pi_j=\pi_{ij}$. Moreover, we identify $\phi\colon \FF_q[t]\to \FF_q$ with $\phi(t)\in \FF_q$.  For an unramified $\phi\colon \FF_q[t]\to \FF_q$ we have that  $\phi(t)\in S$ if and only if $\frobC[\Ocl_N/\Ocl_F]{\phi} \cap \pi_j^{-1}(H_j)\neq \varnothing$.  

    Hence if $\Omega\cap G_{N,1}\neq \varnothing$, then the assertion follows from \eqref{imageofmaps}. Otherwise, $\phi(t)\not\in S$. In the latter case, only ramified $\phi$ may satisfy $\phi(t)\in S$, and hence $\#S=O_M(1)$. 
\end{proof}

\begin{remark}\label{rem:dicotomy}
    The proof of \autoref{thm:dicotomy} gives that for large $q$ we have 
    $\#S\gg_M q$ if and only if $\Omega\cap G_{N,1}\neq \varnothing$.
\end{remark}

    We give a geometric sufficient condition for $\#S$ to be large.

\begin{proposition}
    Let $M>0$, let $k\leq M$, let 
     $f_j\in \FF_q[x]$ be polynomials of positive degree $\leq M$, $j=1,\ldots, k$. Let $C_j = \{ f_j(x)-t\}$ and $\pi_j\colon C_j\to \AA^1$ the projection $\pi_j(x,t) = t$. Assume that the fiber product $D=C_1 \times_{\AA^1} \dots \times_{\AA^1} C_k$ is absolutely irreducible. Then,
     \[
        \#  \bigcap_{j=1}^k f_{j}(\FF_q) \gg_M q.
     \]
\end{proposition}

\begin{proof}
    Let $\pi\colon D\to \AA^{1}$ be the natural projection map. Then, $\deg \pi = \prod \deg\pi_j = \prod \deg f_j\ll_M1$. Since $\pi$ factors through $\pi_j$, for each $j$, we have  $\pi(D(\FF_q)) \subseteq \bigcap_j \pi_j(C_j(\FF_q))=\bigcap_j f_{j}(\FF_q)$.
    It remains to prove that $\pi(\#D(\FF_q)) \gg_Mq$. And indeed, 
    by the Lang-Weil estimates, $\#D(\FF_q) = q+O_{M}(q^{-1/2})$. 
    Hence, $\# \pi(D(\FF_q)) \geq \frac{1}{\deg \pi} q+O_{M}(q^{-1/2})\gg_Mq$. 
\end{proof}


\section{Comparison with Entin's theorem}
\label{sec:comp}

For a set $\hat{S} \subseteq \FF_q^n$, the irregularity of $\hat{S}$ as defined in \cite[eq. (4)]{entin2021factorization} is:
\begin{equation}
  \label{eq:irregularity-of-set}
  \irreg\pa*{\hat{S}} = \frac{q^n}{\#\hat{S}} \sum_{\vecb \in \FF_q^n} \abs*{\widehat{\onebb_{\hat{S}} }\pa*{\vecb}}.
\end{equation}

Let $S = \{\alpha^2 | \alpha\in \FF_q\}$.
We compute  $\irreg\pa*{S^n}$ 
and show that
  \begin{equation}
    \label{thm:irregularity-of-squares}
    \irreg\pa*{S^n} \ge \pa*{\sqrt{q}-1}^n.
  \end{equation}
  This implies that the error term in Entin's theorem \eqref{eq:entin-factorization} is $\gg q^{(n-1)/2}$, hence it does not imply \autoref{thm:main-squares-coefficients}. 
\begin{proof}[Computation of \eqref{thm:irregularity-of-squares}]
  By \cite[Lemma 5.1]{entin2021factorization}, $\irreg\pa*{S^n}=\irreg\pa*{S}^n$. Hence, its suffice to show  that $\irreg\pa*{S} \ge \sqrt{q}-1$.
  
  Consider the quadratic multiplicative character $\chi_2\colon \FF_q \to \CC$ defined by
  \begin{equation}
    \chi_2\pa*{\alpha} = \begin{cases}
      1 & \text{if } \exists \beta \in \FF_q^\times \text{ such that } \beta^2 = \alpha, \\
      0 & \text{if } \alpha = 0, \\
      -1 & \text{otherwise}.
    \end{cases}
  \end{equation}
  So for $\alpha\neq 0$ we have that $\onebb_{S}\pa*{\alpha} = \pa*{1 + \chi_2\pa*{\alpha}}/2$.
  Thus, for $\beta \in \FF_q^\times$, 
  \begin{multline}
    \widehat{\onebb_{S}}\pa*{\beta}
    = \frac{1}{q} \sum_{\alpha \in \FF_q} \onebb_{S}\pa*{\alpha} e_q\pa*{\alpha\beta} 
    = \frac{1}{q} \sum_{\alpha \in \FF_q^\times} \frac{1 + \chi_2\pa*{\alpha}}{2} \cdot e_q\pa*{\alpha\beta} + \frac{1}{q}\\
    = \frac{1}{2q} \sum_{\alpha \in \FF_q^\times} e_q\pa*{\alpha\beta} + \frac{1}{2q} \sum_{\alpha \in \FF_q^\times} \chi_2\pa*{\alpha} e_q\pa*{\alpha\beta} + \frac{1}{q}.
  \end{multline}
  By orthogonality of characters,  $\sum_{\alpha \in \FF_q^\times} e_q\pa*{\alpha\beta}=-1$.
  By the classical theory of Gauss sums (see e.g.\ \cite[Propostion 8.2.2]{ireland1990classical}), 
  \begin{equation}
    \abs*{\sum_{\alpha \in \FF_q^\times} \chi_2\pa*{\alpha} e_q\pa*{\alpha\beta}}  =\sqrt{q}.
  \end{equation}
    Thus by the reverse triangle inequality, we have that
  \begin{equation}\label{eq:bdd}
    \abs*{\widehat{\onebb_{S}}\pa*{\beta}} \ge
    \frac{1}{2\sqrt{q}} - \frac{1}{q}=\frac{\sqrt{q}-1}{2q}.
  \end{equation}

  Finally,  since $\#S =\pa*{q+1}/2$, we substitute \eqref{eq:bdd} in \eqref{eq:irregularity-of-set} to get that
  \begin{multline}
      \irreg\pa*{S} 
      = \frac{2q}{q+1} \sum_{\beta \in \FF_q} \abs*{\widehat{\onebb_{S}}\pa*{\beta}} \\
      = \frac{2q}{q+1} \sum_{\beta \in \FF_q^{\times}} \abs*{\widehat{\onebb_{S}}\pa*{\beta}} + \frac{2q}{q+1}
       \geq \frac{2q}{q+1} \frac{(q-1)(\sqrt{q}-1)}{2q} + \frac{2q}{q+1} 
       \\=\sqrt{q}-1+\frac{2q-2\sqrt{q}+2}{q+1}
      \geq \sqrt{q}-1. \qedhere
  \end{multline}
\end{proof}

\bibliographystyle{alpha}
\bibliography{main}

\end{document}